\documentclass[12pt,a4paper]{amsart}
\usepackage{url}
\usepackage[colorlinks,linkcolor=blue,citecolor=blue]{hyperref}

\usepackage{amsmath,amsthm,amsfonts}

\newtheorem{theorem}{Theorem}

\newtheorem{proposition}{Proposition}

\theoremstyle{definition}

\newtheorem{definition}{Definition}

\newtheorem{remark}{Remark}

\newcommand{\R}{\mathbb{R}}

\usepackage[T1]{fontenc}
\usepackage{cmbright}

\begin{document}

\title{Bernstein Polynomials on simplex}
 \author{A. Bayad}
\address{Abdelmejid Bayad. D\'epartement de math\'ematiques \\
Universit\'e d'Evry Val d'Essonne, Bd. F. Mitterrand, 91025 Evry
Cedex, France  \\} \email{abayad@maths.univ-evry.fr}
\author{T. Kim}
\address{Taekyun Kim. Division of General Education-Mathematics \\
Kwangwoon University, Seoul 139-701, Republic of Korea  \\}
\email{tkkim@kw.ac.kr}
\author{S.-H. Rim}
\address{Seog-Hoon Rim. Department of Mathematics Education, \\
Kyungpook National University, Taegu 702-701, Republic of Korea  \\}
\email{shrim@knu.ac.kr}


\date{\today}
\begin{abstract} We prove two identities for multivariate Bernstein
  polynomials on simplex, which  are considered on a pointwise.
In this paper, we study good approximations of Bernstein polynomials
for every continuous function on simplex and the higher dimensional
$q$-analogue of Bernstein polynomials on simplex.
\end{abstract}

\maketitle

\allowdisplaybreaks

\section{Introduction and motivation}
Recently many mathematicians study on the theory  of multivariate Bernstein
 polynomials on simplex . This   theory has many applications in
  different areas in mathematics and physics,
see [1--17]. \\
 Throughout this paper we set
$I = [0,1]$ and $k\in\mathbb N.$ Taking a $k$-dimensional simplex $\Delta_k:$
\[
\Delta_k=\{ \stackrel{\rightarrow}{{ x}}=(x_1,\ldots,x_k)\in I^k:~~x_1+\ldots+x_k\leq 1\}.
\]
As well-known, Bernstein polynomials (\ref{ordi-Bernstein}) are the
most important and interesting concrete operators on a space of
continuous functions(see [15,16]). The purpose of this paper is to
study their generalization to $k$-dimensional simplex.
\begin{definition} The $n$-th degree ordinary Bernstein polynomial $B_{v,n}:I \to \R$ is given by
\begin{eqnarray}\label{ordi-Bernstein}
    B_{v,n}(x) = \binom{n}{v} \, x^v (1- x)^{n-v},
\end{eqnarray}
$v = 0,1,\ldots,n$.  We extend this
to $$B_{\stackrel{\rightarrow}{{ v}},n}:\Delta_k \to \R$$ by
taking $\stackrel{\rightarrow}{{ v}}$ to be a multi-index,
$\stackrel{\rightarrow}{{ v}} = (v_1,\ldots,v_k)\in\mathbb
N_0^k$, define $$|\stackrel{\rightarrow}{{ v}}|:=v_1+\ldots+v_k \in \{0,1,\ldots,n\}$$ and setting
\begin{eqnarray}\label{k-dim-Bernstein}
B_{\stackrel\rightarrow{v},n}(\stackrel{\rightarrow}{{x}}) = \binom{n}{\stackrel{\rightarrow}{{ v}}} \stackrel{\rightarrow}{{ x}}^{\stackrel{\rightarrow}{{ v}}}(1-|\stackrel{\rightarrow}{{ x}}|)^{n-|\stackrel{\rightarrow}{{ v}}|},
\end{eqnarray}
where
\[
\stackrel{\rightarrow}{{ x}} = (x_1, \ldots, x_k) \in\Delta_k,
\stackrel{\rightarrow}{{ x}}^{\stackrel{\rightarrow}{{
      v}}}=\prod_{i=1}^kx_i^{v_i},~\stackrel{\rightarrow}{{ v}}! = v_1!\cdots v_k! ~~\textrm{ and }
 \binom{n}{\stackrel{\rightarrow}{{ v}}}=\frac{n!}{{\stackrel{\rightarrow}{{ v}}!}(n-|\stackrel{\rightarrow}{{ v}}|)!}.
\]
\end{definition}
For every $f$  defined on $\Delta_k$, we write
\begin{eqnarray}
\mathbb B_n(f|\stackrel{\rightarrow}{{x}})=\sum_{|\stackrel{\rightarrow}{{ v}}|\leq n}f({\stackrel{\rightarrow}{{ v}}}/{n})B_{\stackrel\rightarrow{v},n}(\stackrel{\rightarrow}{{x}}).
\end{eqnarray}

Here we have  the convergence
\begin{proposition}\label{conv}
    If $f:\Delta_k \to \R$ is continuous, then $\mathbb B_n(f|.) \to f$ uniformly on $\Delta_k$ as $n \to \infty$.
\end{proposition}
\begin{proof}
The proof of this proposition \ref{conv} is quite simple and is based
on the partition property of the
$B_{\stackrel\rightarrow{v},n}(\stackrel\rightarrow{x})$ that
\[\sum_{|\stackrel\rightarrow{v}|\leq
  n}B_{\stackrel\rightarrow{v},n}(\stackrel{\rightarrow}{{ x}})=1.\]
 This proof is similar to
that Theorem 1.1.1 in \cite[p.5--8]{lorentz}.
 Since $f$ is uniformly continuous on $\Delta_k$, for given $\epsilon
 > 0$, there exists $\delta > 0$ with the property that if
 $\stackrel\rightarrow{x} = (x_1, \ldots, x_k)$,
 $\stackrel\rightarrow{y} = (y_1, \ldots, y_k)$, and $|x_i -
 y_i|<\delta$ for all $i$, then $|f(\stackrel\rightarrow{x}) -
 f(\stackrel\rightarrow{y})| < \epsilon$.  We define the distance on
 $\Delta_k$ by
 $d(\stackrel\rightarrow{x},\stackrel\rightarrow{y}) = \max \{|x_1 -
 y_1|, \ldots, |x_k - y_k|\}$.
We suppose that $n$ is sufficiently large that $\frac{1}{4 n\delta^2} < \epsilon$.
We have
\[
|\mathbb B_n(f|\stackrel\rightarrow{x}) - f(\stackrel\rightarrow{x})| \leq \underset{d\left( \frac{\stackrel\rightarrow{v}}{n},\stackrel\rightarrow{x}\right) < \delta}{\sum} |f(\tfrac{\stackrel\rightarrow{v}}{n})-f(\stackrel\rightarrow{x})| \, B_{\stackrel\rightarrow{v},n}(\stackrel\rightarrow{x}) + \underset{d\left( \frac{\stackrel\rightarrow{v}}{m},\stackrel\rightarrow{x}\right) \geq \delta}{\sum} |f(\tfrac{\stackrel\rightarrow{v}}{m})-f(\stackrel\rightarrow{x})| \, B_{\stackrel\rightarrow{v},n}(\stackrel\rightarrow{x})
\]
where $\stackrel\rightarrow{v} = (v_1,\ldots,v_k)$, a multi-index.  The first sum satisfies
\[\underset{d\left( \frac{\stackrel\rightarrow{v}}{m},\stackrel\rightarrow{x}\right) < \delta}{\sum} |f(\tfrac{\stackrel\rightarrow{v}}{n})-f(\stackrel\rightarrow{x})| \, B_{\stackrel\rightarrow{v},n}(\stackrel\rightarrow{x}) < \epsilon
\]
by uniform continuity of $f$.
Now we study the second sum. Let $M = \max f$.  Then
\[
     \underset{d\left( \frac{\stackrel\rightarrow{v}}{n},\stackrel\rightarrow{x}\right) \geq \delta}{\sum} |f(\tfrac{\stackrel\rightarrow{v}}{m})
     -f(\stackrel\rightarrow{x})| \, B_{\stackrel\rightarrow{v},n}(\stackrel\rightarrow{x})\; \leq    \; 2  M  \underset{d\left( \frac{\stackrel\rightarrow{v}}{n},x\right) \geq \delta}{\sum}  B_{\stackrel\rightarrow{v},n}(\stackrel\rightarrow{x}).
\]
We see that
\begin{gather*}
    \underset{d\left(
          \frac{\stackrel\rightarrow{v}}{n},\stackrel\rightarrow{x}\right)
          \geq \delta}{\sum}
        B_{\stackrel\rightarrow{v},n}(\stackrel\rightarrow{x})< n \epsilon.
\end{gather*}
Thus $|\mathbb B_n(f|\stackrel\rightarrow{x}) - f(\stackrel\rightarrow{x})| < \epsilon + 2Mn\epsilon$, and we are done.
\end{proof}
The generating functions of the $k$-dimensional Bernstein
polynomials are as follows:
\begin{proposition}
\begin{eqnarray}
\sum_{n\geq |\stackrel{\rightarrow}{{ v}}|}B_{\stackrel{\rightarrow}{{ v}},n}(\stackrel{\rightarrow}{{ x}})\frac{t^n}{n!}=\frac{(t\stackrel{\rightarrow}{{ x}})^{\stackrel{\rightarrow}{{ v}}}}{\stackrel{\rightarrow}{{ v}}!}e^{t(1-|\stackrel{\rightarrow}{{ x}}|)}.
\end{eqnarray}
\end{proposition}
\begin{proof}
Writing
\begin{eqnarray}
\sum_{n\geq |\stackrel{\rightarrow}{{
      v}}|}B_{\stackrel{\rightarrow}{{
      v}},n}(\stackrel{\rightarrow}{{
    x}})\frac{t^n}{n!}&=&\sum_{n\geq |\stackrel{\rightarrow}{{
      v}}|} \binom{n}{\stackrel{\rightarrow}{{ v}}}
\stackrel{\rightarrow}{{ x}}^{\stackrel{\rightarrow}{{
      v}}}(1-|\stackrel{\rightarrow}{{
    x}}|)^{n-|\stackrel{\rightarrow}{{ v}}|}\frac{t^n}{n!}\\
&=&\sum_{n\geq |\stackrel{\rightarrow}{{
      v}}|}
\frac{(t\stackrel{\rightarrow}{{ x}})^{\stackrel{\rightarrow}{{
      v}}}}  {\stackrel{\rightarrow}{{ v}}!}
\frac{(1-|\stackrel{\rightarrow}{{x}}|)^{n-|\stackrel{\rightarrow}{{
        v}}|} t^{n-|\stackrel{\rightarrow}{{v}}|} }{ (n-|\stackrel{\rightarrow}{{
        v}}|)!} \\
&=&
\frac{(t\stackrel{\rightarrow}{{x}})^{\stackrel{\rightarrow}{{
        v}}}}
{\stackrel{\rightarrow}{{ v}}!} \sum_{m\geq 0}\frac{(1-|\stackrel{\rightarrow}{{x}}|)^{m} t^{m}}{ m!}.
\end{eqnarray}
This yields the equality
\begin{eqnarray}
\sum_{n\geq |\stackrel{\rightarrow}{{
      v}}|}B_{\stackrel{\rightarrow}{{
      v}},n}(\stackrel{\rightarrow}{{
    x}})\frac{t^n}{n!}=\frac{(t\stackrel{\rightarrow}{{ x}})^{\stackrel{\rightarrow}{{  v}}}}  {\stackrel{\rightarrow}{{ v}}!} e^{t(1-|\stackrel{\rightarrow}{{ x}}|)}.
\end{eqnarray}
\end{proof}
\section{Main results and proofs}
This section contains the main results of this paper. The first main result can be state as follows.
\begin{theorem}\label{Thm1}
For $n\in\mathbb N,m\in\mathbb N_0$ and $\stackrel{\rightarrow}{{ v}}\in\mathbb N_0^k$ such that $m\leq \textrm{min}\left(|\stackrel{\rightarrow}{{ v}}|,n\right)$. Then we have the following identity
\begin{eqnarray}\label{somme}
\sum_{\stackrel{\rightarrow}{{ u}}\leq \stackrel{\rightarrow}{{ v}}\atop |\stackrel{\rightarrow}{{ u}}|\leq m}\frac{\stackrel{\rightarrow}{{ u}}!(m-|\stackrel{\rightarrow}{{ u}}|)!}{m!}B_{\stackrel{\rightarrow}{{ u}},m}(\stackrel{\rightarrow}{{ x}})B_{\stackrel{\rightarrow}{{ v}}-\stackrel{\rightarrow}{{ u}},n-m}(\stackrel{\rightarrow}{{ x}})=B_{\stackrel{\rightarrow}{{ v}},n}(\stackrel{\rightarrow}{{ x}}),
\end{eqnarray}
where $\stackrel{\rightarrow}{{ u}}\leq
\stackrel{\rightarrow}{{ v}}$ means that $0\leq u_i\leq v_i$ for all $i=1,\cdots, k.$
\end{theorem}

The formula (\ref{somme}) can be viewed as a pointwise recurrence or
orthogonality formula for the Bernstein polynomials in
$k$-dimensional simplex.
\begin{remark}
For $m=1$, we obtain from Theorem \ref{Thm1} the following recurrence
formula
\begin{eqnarray}
(1-|\stackrel{\rightarrow}{{ x}}|)
  B_{\stackrel{\rightarrow}{{
        v}},{n-1}}(\stackrel{\rightarrow}{{ x}})+
\sum_{|\stackrel{\rightarrow}{{ u}}|=1\atop
\stackrel{\rightarrow}{{ u}}
\leq \stackrel{\rightarrow}{{ v}}}
\stackrel{\rightarrow}{{ x}}^{\stackrel{\rightarrow}{{
      u}}} B_{\stackrel{\rightarrow}{{v}}-\stackrel{\rightarrow}{{
        u}},n-1}(\stackrel{\rightarrow}{{ x}})=
B_{\stackrel{\rightarrow}{{ v}},n}(\stackrel{\rightarrow}{{ x}})
\end{eqnarray}
\end{remark}
\begin{proof}
 We prove this theorem by induction
on $n$ and $m$.  For $n=0,1$ the statement is trivial. Let $n\geq 2.$ Taking
$m=1.$ The sum
\begin{eqnarray*}& &\sum_{\stackrel{\rightarrow}{{ u}}\leq
    \stackrel{\rightarrow}{{ v}}\atop
    |\stackrel{\rightarrow}{{ u}}|\leq
    m}\frac{\stackrel{\rightarrow}{{
        u}}!(m-|\stackrel{\rightarrow}{{
        u}}|)!}{m!}B_{\stackrel{\rightarrow}{{
        u}},m}(\stackrel{\rightarrow}{{
      x}})B_{\stackrel{\rightarrow}{{
        v}}-\stackrel{\rightarrow}{{
        u}},n-m}(\stackrel{\rightarrow}{{ x}})\\
        &=&
B_{\stackrel{\rightarrow}{{
        0}},1}(\stackrel{\rightarrow}{{
      x}})B_{\stackrel{\rightarrow}{{
        v}},n-1}(\stackrel{\rightarrow}{{ x}})+
\sum_{\stackrel{\rightarrow}{{ u}}\leq
    \stackrel{\rightarrow}{{ v}}\atop
    |\stackrel{\rightarrow}{{ u}}|=1} B_{\stackrel{\rightarrow}{{
        u}},1}(\stackrel{\rightarrow}{{
      x}})B_{\stackrel{\rightarrow}{{
        v}}-\stackrel{\rightarrow}{{
        u}},n-1}(\stackrel{\rightarrow}{{ x}})\\
&=&(1-|\stackrel{\rightarrow}{{ x}}|)
  B_{\stackrel{\rightarrow}{{
        v}},{n-1}}(\stackrel{\rightarrow}{{ x}})+
\sum_{|\stackrel{\rightarrow}{{ u}}|=1\atop
\stackrel{\rightarrow}{{ u}}
\leq \stackrel{\rightarrow}{{ v}}}
\stackrel{\rightarrow}{{ x}}^{\stackrel{\rightarrow}{{
      u}}} B_{\stackrel{\rightarrow}{{v}}-\stackrel{\rightarrow}{{
        u}},n-1}(\stackrel{\rightarrow}{{ x}}).
\end{eqnarray*}
By using the following fact $|\stackrel{\rightarrow}{{ u}}|=1$ if
and only if one index of $\stackrel{\rightarrow}{{ u}}$ is $1$ and
all the others are zero, after simple manipulation,  we obtain the
relation
\begin{eqnarray*}
(1-|\stackrel{\rightarrow}{{ x}}|)
  B_{\stackrel{\rightarrow}{{
        v}},{n-1}}(\stackrel{\rightarrow}{{ x}})+
\sum_{|\stackrel{\rightarrow}{{ u}}|=1\atop
\stackrel{\rightarrow}{{ u}}
\leq \stackrel{\rightarrow}{{ v}}}
\stackrel{\rightarrow}{{ x}}^{\stackrel{\rightarrow}{{
      u}}} B_{\stackrel{\rightarrow}{{v}}-\stackrel{\rightarrow}{{
        u}},n-1}(\stackrel{\rightarrow}{{ x}})=
B_{\stackrel{\rightarrow}{{ v}},n}(\stackrel{\rightarrow}{{ x}}).
\end{eqnarray*}
Then the Theorem \ref{Thm1} is valid for any $n$ and $m=1$. Now we
suppose the theorem  holds up to $n\geq m\geq 1$. We can write for
$n+1$  and $m=1$ the following
\begin{eqnarray*}
B_{\stackrel{\rightarrow}{{ v}},n+1}(\stackrel{\rightarrow}{{ x}})=
\sum_{\stackrel{\rightarrow}{{ u}}\leq
  \stackrel{\rightarrow}{{ v}}\atop
  |\stackrel{\rightarrow}{{ u}}|\leq
  1}\frac{\stackrel{\rightarrow}{{
      u}}!(1-|\stackrel{\rightarrow}{{
      u}}|)!}{1!}B_{\stackrel{\rightarrow}{{
      u}},1}(\stackrel{\rightarrow}{{
    x}})B_{\stackrel{\rightarrow}{{
      v}}-\stackrel{\rightarrow}{{
      u}},n}(\stackrel{\rightarrow}{{ x}})
\end{eqnarray*}
Then we get
\begin{eqnarray}\label{ident1}
B_{\stackrel{\rightarrow}{{ v}},n+1}(\stackrel{\rightarrow}{{ x}})=\sum_{\stackrel{\rightarrow}{{ u}}+ \leq
  \stackrel{\rightarrow}{{ v}}\atop
  |\stackrel{\rightarrow}{{ u}}|\leq
  1}B_{\stackrel{\rightarrow}{{
      u}},1}(\stackrel{\rightarrow}{{
    x}})B_{\stackrel{\rightarrow}{{
      v}}-\stackrel{\rightarrow}{{
      u}},n}(\stackrel{\rightarrow}{{ x}}),
\end{eqnarray}
by using the recurrence hypothesis, from (\ref{ident1}), we obtain
\begin{eqnarray*}
B_{\stackrel{\rightarrow}{{
      v}},n+1}(\stackrel{\rightarrow}{{ x}})&=&
\sum_{\stackrel{\rightarrow}{{ u}}\leq
  \stackrel{\rightarrow}{{ v}}\atop
  |\stackrel{\rightarrow}{{ u}}|\leq
  1}B_{\stackrel{\rightarrow}{{
      u}},1}(\stackrel{\rightarrow}{{
    x}})
\sum_{\stackrel{\rightarrow}{{ u^{\prime}}}\leq
  \stackrel{\rightarrow}{{ v}}-\stackrel{\rightarrow}{{
      u}}\atop |\stackrel{\rightarrow}{{ u^{\prime}}}|\leq
  m}\frac{\stackrel{\rightarrow}{{
      u^{\prime}}}!(m-|\stackrel{\rightarrow}{{
      u^{\prime}}}|)!}{m!}B_{\stackrel{\rightarrow}{{
      u^{\prime}}},m}(\stackrel{\rightarrow}{{
    x}})B_{\stackrel{\rightarrow}{{
      v}}-\stackrel{\rightarrow}{{u}}-\stackrel{\rightarrow}{{
      u^{\prime}}},n-m}(\stackrel{\rightarrow}{{ x}})\\
&=& \sum_{\stackrel{\rightarrow}{{ u}}+\stackrel{\rightarrow}{{ u^{\prime}}}\leq
  \stackrel{\rightarrow}{{ v}}\atop
  |\stackrel{\rightarrow}{{ u}}|\leq
  1 , |\stackrel{\rightarrow}{{ u^{\prime}}}|\leq m
}
\frac{\stackrel{\rightarrow}{{
      u^{\prime}}}!(m-|\stackrel{\rightarrow}{{
      u^{\prime}}}|)!}{m!}B_{\stackrel{\rightarrow}{{
      u}},1}(\stackrel{\rightarrow}{{
    x}})
B_{\stackrel{\rightarrow}{{
      u^{\prime}}},m}(\stackrel{\rightarrow}{{
    x}})B_{\stackrel{\rightarrow}{{
      v}}-\stackrel{\rightarrow}{{u}}-\stackrel{\rightarrow}{{
      u^{\prime}}},n-m}(\stackrel{\rightarrow}{{ x}}).
\end{eqnarray*}
Setting $\stackrel{\rightarrow}{{
      w}}= \stackrel{\rightarrow}{{
    u}}+\stackrel{\rightarrow}{{ u^{\prime}}}.$ From the relation (2), we deduce the identity
\begin{eqnarray*}
B_{\stackrel{\rightarrow}{{ v}},n+1}(\stackrel{\rightarrow}{{ x}})=
\sum_{\stackrel{\rightarrow}{{ w}}\leq \stackrel{\rightarrow}{{ v}}\atop |\stackrel{\rightarrow}{{ w}}|\leq m+1}\frac{\stackrel{\rightarrow}{{w}}!(m+1-|\stackrel{\rightarrow}{{ w}}|)!}{(m+1)!}B_{\stackrel{\rightarrow}{{ w}},m+1}(\stackrel{\rightarrow}{{ x}})B_{\stackrel{\rightarrow}{{ v}}-\stackrel{\rightarrow}{{ w}},n-m}(\stackrel{\rightarrow}{{ x}}).
\end{eqnarray*}
This completes the proof of the theorem.
\end{proof}
 For every $j,m, 1\leq j\leq k, m\geq 1$, we define the affine transformations $T_{j,m}$ by
\begin{eqnarray}\label{transformations}
T_{j,m}(\stackrel{\rightarrow}{{ x}})=(x_1,\cdots,x_{j-1},m-|\stackrel{\rightarrow}{{ x}}|,x_{j+1},\cdots,x_k).
\end{eqnarray}
and for every $\sigma\in{\mathcal S}_k$ permutation of the set $\{1,\cdots, k\}$ we put
\begin{eqnarray}
\sigma(\stackrel{\rightarrow}{{
    x}})=(x_{\sigma(1)},\cdots,(x_{\sigma(k)}).
\end{eqnarray}
We state now the second main result of this paper.
\begin{theorem}\label{Thm2}
For $n\in\mathbb N$ and $\stackrel{\rightarrow}{{ v}}\in\mathbb N_0^k$.
Then we have the following identities
\begin{eqnarray}\label{symetrie}
B_{\stackrel{\rightarrow}{{ v}},n}(T_{j,1}(\stackrel{\rightarrow}{{ x}}))=B_{T_{j,n}(\stackrel{\rightarrow}{{ v}}),n}(\stackrel{\rightarrow}{{ x}}),
\end{eqnarray}
and
\begin{eqnarray}\label{eq2}
B_{\stackrel{\rightarrow}{{ v}},n}(\sigma(\stackrel{\rightarrow}{{ x}}))=B_{\sigma^{-1}(\stackrel{\rightarrow}{{ v}}),n}(\stackrel{\rightarrow}{{ x}}).
\end{eqnarray}
\end{theorem}
 The relation (\ref{symetrie}) is a multivariate symmetry
formula for the Bernstein polynomilas.
\begin{remark} Taking $j=1$ and $\sigma=(12)$ we get from the Theorem
  \ref{Thm2} the symetries relations
\begin{eqnarray}
B_{\stackrel{\rightarrow}{{
      v}},n}(1-|\stackrel{\rightarrow}{{x}}|,x_2,\cdots,x_k)=B_{(n-|\stackrel{\rightarrow}{{
      v}}| , v_2 , \cdots,v_k),n}(\stackrel{\rightarrow}{{ x}}),
\end{eqnarray}
and
\begin{eqnarray}B_{\stackrel{\rightarrow}{{ v}},n}(x_2,x_1,x_3,\cdots,x_k)=B_{(v_2,v_1,x_3,\cdots,v_k),n}(\stackrel{\rightarrow}{{ x}}).
\end{eqnarray}
\end{remark}
\begin{proof} By using the equalities (2) and
  (\ref{transformations}) we have
\begin{eqnarray*}
B_{\stackrel{\rightarrow}{{ v}},n}(T_{j,1}(\stackrel{\rightarrow}{{ x}}))&=&
B_{\stackrel{\rightarrow}{{
      v}},n}(x_1,\cdots,x_{j-1} ,\cdots,
1-|\stackrel{\rightarrow}{{ x}}|,x_{j+1},\cdots,x_k)\\
&=&\binom{n}{\stackrel{\rightarrow}{{ v}}}x_1^{v_1}\cdots
x_{j-1}^{v_{j-1}}(1-|\stackrel{\rightarrow}{{
    x}}|)^{v_j}x_{j+1}\cdots x_k^{v_k} x_j^{
  1-|\stackrel{\rightarrow}{{ v}}|}\\
&=&\binom{n}{\stackrel{\rightarrow}{{ v}}}x_1^{v_1}\cdots
x_{j-1}^{v_{j-1}}x_j^{ n-|\stackrel{\rightarrow}{{ v}}|}x_{j+1}\cdots x_k^{v_k}
(1-|\stackrel{\rightarrow}{{
    x}}|)^{v_j} .
\end{eqnarray*}
This implies, by using the relation (\ref{transformations}), the identity
\begin{eqnarray*}
B_{\stackrel{\rightarrow}{{ v}},n}(T_{j,1}(\stackrel{\rightarrow}{{ x}}))=B_{T_{j,n}(\stackrel{\rightarrow}{{ v}}),n}(\stackrel{\rightarrow}{x}).
\end{eqnarray*}
The  equality (\ref{eq2})  of the Theorem \ref{Thm2} can be obtained
in a similar way of (\ref{symetrie}).
\end{proof}
\section{ $q$-extension of Bernstein polynomials on simplex}

When one talks of $q$-extension, $q$ is variously considered  as
an indeterminate, a complex number $q\in \mathbb{C},$ or $p$-adic
number $q\in\Bbb C_p .$ If $q\in \Bbb C$, then we always assume
that $|q|<1.$  If $q\in  \mathbb{C}_p,$ we usually assume that
$|1-q|_p< 1$. Here, the symbol $| \cdot|_p $ stands for the
$p$-adic absolute value on  $\mathbb{C}_p$ with $|p|_p \leq
{1}/{p}.$ For each $x$, the $q$-basic numbers are defined by
$$[x]_q =\frac{1-q^x}{1-q}. $$ 
 We extend this by 
$$[\stackrel{\rightarrow}{{ x}}]_q =([x_1]_q,\ldots,[x_k]_q)$$

and the $q$-extension of Bernstein polynomials on $\Delta_k$ is defined by
\begin{eqnarray}\label{k-dim-Bernstein}
B_{\stackrel\rightarrow{v},n}(\stackrel{\rightarrow}{{x}}\mid q) = \binom{n}{\stackrel{\rightarrow}{{ v}}} [\stackrel{\rightarrow}{{ x}}^{\stackrel{\rightarrow}{{ v}}}]_q[1-|\stackrel{\rightarrow}{{ x}}|]_q^{n-|\stackrel{\rightarrow}{{ v}}|}.
\end{eqnarray}
Here again we have the $q$-extensions of Theorem \ref{Thm1} and Theorem \ref{Thm2}.
\begin{theorem}\label{Thm3}
For $n\in\mathbb N,m\in\mathbb N_0$ and $\stackrel{\rightarrow}{{ v}}\in\mathbb N_0^k$ such that $m\leq \textrm{min}\left(|\stackrel{\rightarrow}{{ v}}|,n\right)$. Then we have the following identity
\begin{eqnarray*}\label{somme2}
\sum_{\stackrel{\rightarrow}{{ u}}\leq \stackrel{\rightarrow}{{ v}}\atop |\stackrel{\rightarrow}{{ u}}|\leq m}\frac{\stackrel{\rightarrow}{{ u}}!(m-|\stackrel{\rightarrow}{{ u}}|)!}{m!}B_{\stackrel{\rightarrow}{{ u}},m}(\stackrel{\rightarrow}{{ x}}|q)B_{\stackrel{\rightarrow}{{ v}}-\stackrel{\rightarrow}{{ u}},n-m}(\stackrel{\rightarrow}{{ x}}|q)=B_{\stackrel{\rightarrow}{{ v}},n}(\stackrel{\rightarrow}{{ x}}|q),
\end{eqnarray*}
where $\stackrel{\rightarrow}{{ u}}\leq
\stackrel{\rightarrow}{{ v}}$ means that $0\leq u_i\leq v_i$ for all $i=1,\cdots, k.$
\end{theorem}

\begin{theorem}\label{Thm4}
For $n\in\mathbb N$ and $\stackrel{\rightarrow}{{ v}}\in\mathbb N_0^k$. Then we have the following identities
\begin{eqnarray}\label{symetrie3}
B_{\stackrel{\rightarrow}{{ v}},n}(T_{j,1}(\stackrel{\rightarrow}{{ x}})|q)=B_{T_{j,n}(\stackrel{\rightarrow}{{ v}}),n}(\stackrel{\rightarrow}{{ x}}|q),
\end{eqnarray}
and
\begin{eqnarray}\label{eq3}
B_{\stackrel{\rightarrow}{{ v}},n}(\sigma(\stackrel{\rightarrow}{{ x}})|q)=B_{\sigma^{-1}(\stackrel{\rightarrow}{{ v}}),n}(\stackrel{\rightarrow}{{ x}}|q).
\end{eqnarray}
\end{theorem}
The proofs of these theorems are quite similar to those of Theorems
\ref{Thm1} and \ref{Thm2}. Then we omit them.
\section{ \bf Acknowledgements}
 The present Research has been conducted by the Research Grant of Kwangwoon University in 2011.


\begin{thebibliography}{99}
\bibitem{Abel} U. Abel, Z. Li, \emph{ A new proof of an identity of
  Jetter and St\"ockler for multivariate Bernstein polynomials},\emph{
  Computer Aided Geometric Design} \textbf{ 23} (2006), pp. 297-301.

\bibitem{AcikgozSerkan} A. Bayad, T. Kim, \emph{ Identities involving  values of Bernstein, $q$-Bernoulli,  and $q$-Euler polynomials},  \emph{Russ. J. Math.Phys.} \textbf{18} (2011), pp.133-143.
  \bibitem{Ding}C. Ding, F. Cao, \emph{ $K$-functionals and
    multivariate Bernstein polynomials},\emph{ Journal of
    Approximation Theory } \textbf{155} (2008), pp.125-135.

 \bibitem{Feng}Y. Y. Feng , J. Kozak , \emph{  Asymptotic expansion
   formula for Bernstein polynomials defined on a simplex},\emph{
   Constructive Approximation}, Volume 8, Number \textbf{1} ,(1992),
   pp. 49-58.

\bibitem{bernstein} S. N. Bernstein, D\'{e}monstration du th\'{e}or\`{e}me
de Weierstrass fond\'{e}e sur la calcul des probabilit\'{e}s. Comm. Soc.
Math. Charkow S\'{e}r. 2 t. 13, 1-2 (1912-1913).

\bibitem{Goldman} L. Bus\'{e} and R. Goldman, Division algorithms for
Bernstein polynomials, Computer Aided Geometric Design, 25(9) (2008),
850-865.

\bibitem{GoldmanBOOK} R. Goldman, An Integrated Introduction to Computer
Graphics and Geometric Modeling, \textit{CRC Press, Taylor and Francis}, New
York, 2009.

\bibitem{GoldmanBook2} R. Goldman, Pyramid Algorithms: A Dynamic Programming
Approach to Curves and Surfaces for Geometric Modeling, Morgan Kaufmann
Publishers, Academic Press, San Diego, 2002.

\bibitem{GoldmanBook3} R. Goldman, Identities for the Univariate and
Bivariate Bernstein Basis Functions, Graphics Gems V, edited by Alan Paeth,
Academic Press, (1995), 149-162.
\bibitem{Jetter} K. Jetter, J. St\"ockler, \emph{   An Identity for
   Multivariate Bernstein Polynomial}, \emph{ Computer Aided Geometric
   Design} \textbf{20} (2003), pp.563-577.

\bibitem{JangSimsek} T. Kim, \emph{A note on $q$-Bernstein polynomials}, \emph{Russ. J. Math.
Phys.} \textbf{18} (2011), pp. 73-82.

\bibitem{KCKR} T. Kim, J. Choi, Y. H. Kim, C. S. Ryoo,  \emph{On the fermionic $p$-adic integral
representation of Bernstein polynomials associated with Euler
numbers and polynomials},\emph{ J. Inequal. Appl.,}  Art. ID
864247, 12 pages, 2010.

\bibitem{KCK} T. Kim, J. Choi, Y. H. Kim,
  \emph{q-Bernstein polynomials associated with $q$-Stirling numbers
  and Carlitz's $q$-Bernoulli numbers,} \emph{ Abstr. Appl. Anal.,}
  Art. ID 150975, 11 pages, 2010.

\bibitem{10} T. Kim,  J. Choi, Y. H.Kim, \textit{ $q$-Bernstein polynomials associated with $q$-Stirling numbers and Carlit's $q$-Bernoulli numbers,} \emph{Abstract and Applied Analysis,} Article ID 294715, 10 pages, 2011.

\bibitem{lorentz} G. G. Lorentz, \emph{ Bernstein Polynomials, Second Ed.}, Chelsea, New York, N. Y., 1986.
\bibitem{phillips-2} G. M. Phillips, Interpolation and approximation by
polynomials, CMS Books in\textit{\ }Mathematics/ Ouvrages de Math\'{e}%
matiques de la SMC, 14. Springer-Verlag, New York,\textit{\ }(2003).
\bibitem{Simsek Acikgoz} Y. Simsek and M. Acikgoz, \emph{ A new generating function of ($q$-)
Bernstein-type polynomials and their interpolation function,}
\emph{ Abstract and Applied Analysis,}  Article ID 769095, 12 pages, 2010.
doi:10.1155/2010/769095.



\end{thebibliography}
\end{document}